\documentclass[12pt,leqno]{amsart}
\usepackage{amssymb,url}
\usepackage{rotating}
\usepackage{pdflscape}
\newcommand{\Z}{{\mathbb{Z}}}
\newcommand{\R}{{\mathbb{R}}}

\newcommand{\fCl}{{\mathfrak{C}_{\operatorname{left}}}}
\newcommand{\fCla}{{\mathfrak{C}^\circ_{\operatorname{left}}}}
\newcommand{\fS}{{\mathfrak{S}}}
\newcommand{\ba}{{\mathbf{a}}}
\newcommand{\bb}{{\mathbf{b}}}
\newcommand{\cC}{{\mathcal{C}}}
\newcommand{\cD}{{\mathcal{D}}}
\newcommand{\cF}{{\mathcal{F}}}
\newcommand{\cH}{{\mathcal{H}}}
\newcommand{\cI}{{\mathcal{I}}}
\newcommand{\cO}{{\mathcal{O}}}
\newcommand{\cR}{{\mathcal{R}}}
\newcommand{\cS}{{\mathcal{S}}}
\newcommand{\Hom}{{\operatorname{Hom}}}
\newcommand{\Irr}{{\operatorname{Irr}}}
\newcommand{\prI}{{\operatorname{pr}}}
\renewcommand{\leq}{\leqslant}
\renewcommand{\geq}{\geqslant}

\newtheorem{thm}{Theorem}[section]
\newtheorem{lem}[thm]{Lemma}
\newtheorem{conj}[thm]{Conjecture}
\newtheorem{cor}[thm]{Corollary}
\newtheorem{prop}[thm]{Proposition}

\theoremstyle{definition}
\newtheorem{exmp}[thm]{Example}
\newtheorem{defn}[thm]{Definition}

\theoremstyle{remark}
\newtheorem{rem}[thm]{Remark}

\begin{document}

\title{On the Kazhdan--Lusztig cells in type $E_8$}

\date{\today}

\author{Meinolf Geck and Abbie Halls}
\address{M.G.: Fachbereich Mathematik, IAZ - Lehrstuhl f\"ur Algebra, 
Universit\"at Stuttgart, Pfaf\-fen\-wald\-ring 57, 70569 Stuttgart, Germany}
\email{meinolf.geck@mathematik.uni-stuttgart.de}
\address{A.H.: 21 Rubislaw Terrace Lane, Aberdeen AB10 1XF, UK}
\email{halls.abbie@gmail.com}


\begin{abstract}
In 1979, Kazhdan and Lusztig introduced the notion of ''cells'' (left,
right and two-sided) for a Coxeter group $W$, a concept with numerous 
applications in Lie theory and around. Here, we address algorithmic 
aspects of this theory for finite $W$ which are important in applications,
e.g., run explicitly through all left cells, determine the values of 
Lusztig's $\ba$-function, identify the characters of left cell 
representations. The aim is to show how type $E_8$ (the largest group of 
exceptional type) can be handled systematically and efficiently, too. This 
allows us, for the first time, to solve some open questions in this case, 
including Kottwitz' conjecture on left cells and involutions. Further 
experiments suggest a characterisation of left cells, valid for any 
finite $W$, in terms of Lusztig's $\ba$-function and a slight 
modification of Vogan's generalized $\tau$-invariant.
\end{abstract}

\maketitle

\pagestyle{myheadings}
\markboth{Geck and Halls}{Kazhdan--Lusztig cells in type $E_8$}

\section{Introduction} \label{sec:intro}

Let $W$ be a Coxeter group. Kazhdan and Lusztig \cite{KL} introduced 
certain polynomials $P_{y,w}\in\Z[v]$ (where $y,w\in W$ and $v$ is an 
indeterminate), which have many remarkable properties and appear in a 
number of problems in the representation theory of Lie algebras and 
algebraic groups. As far as $W$ itself is concerned, the polynomials
$P_{y,w}$ give rise to partitions of $W$ into left, right and two-sided 
''cells''. These play an important role, for example, in the classification 
of the irreducible characters of reductive groups over finite fields 
\cite{LuBook}. Since the appearance of \cite{KL}, various authors have 
contributed to the programme of determining the Kazhdan--Lusztig cells for 
certain types of $W$; see, for example, the comments in \cite[\S 4]{Fokko0},
\cite[7.12 and 7.15]{Hump} and \cite[\S 1.7]{Shi}. The situation for 
finite Coxeter groups is as follows. 

For type $A_n$ (when $W$ is a symmetric group), the cells are determined 
in terms of the Knuth--Robinson--Schensted correspondence (see \cite{KL}, 
\cite{Lu1} and also \cite{mymur}). For the classical types $B_n$ and $D_n$, 
see the work of Garfinkle \cite{gar3} and re-interpretations of this work by
Bonnaf\'e et al. \cite{bgil}. For type $I_2(m)$ (when $W$ is a dihedral 
group), the cells are easily determined by explicit computation (by hand); 
see \cite[7.15]{Hump}. For the groups of exceptional type, the determination 
of the cells heavily relies on computers. See Alvis \cite{Al} (types $H_3$, 
$H_4$), Takahashi \cite{taka} ($F_4$), Tong \cite{tong} ($E_6$),  
Chen--Shi \cite{chenshi} ($E_7$) and Chen \cite{chene8} ($E_8$). 

The above results on the various types of groups raise the question if it 
is possible to characterise and to work with the cells in 
terms of some general principles. For applications and experiments, we 
would also like to be able to reconstruct explicitly the cell partition of 
$W$ in a systematic and efficient way. Furthermore, we typically need to 
know more than just the partition of the set $W$ into cells. One of the most 
important aspects of the theory is that every left cell gives rise to a 
representation of $W$; so, for example, given an element $w\in W$, we would 
like to be able to determine the left cell $\Gamma$ containing $w$ and to 
identify the representation afforded by $\Gamma$ (its dimension, its 
character). 

The methods developed in \cite{pycox} allow us to deal with 
questions of this kind for any finite $W$ of rank up to around $8$, with the 
exception of type $E_8$. The group of type $E_8$ has $696729600$ elements; 
it is---by far---the computationally most challenging case among the finite 
Coxeter groups of exceptional type. (Note also that the results in 
\cite{VoE8} are concerned with a modification of the polynomials $P_{y,w}$ 
due to Lusztig--Vogan, so they do not help us here; the highly efficient 
methods of DuCloux \cite{Fokko0}, \cite{Fokko1} are not sufficient either 
for the kind of questions that we are addressing here.) 

The main purpose of this paper is to show how type $E_8$ can be dealt with
efficiently, too. The new algorithms are designed to work with any finite 
$W$ as input (not just type $E_8$ particularly); they are freely available 
in the latest version of the computer algebra package {\sf PyCox} 
\cite{pyc14}. The original motivation for this work was a conjecture due to 
Kottwitz \cite{kottwitz}, concerning the characters of left cell 
representations and intersections of left cells with conjugacy classes of 
involutions.  By work of Kottwitz himself, Casselman \cite{cass}, Bonnaf\'e 
and the first-named author \cite{boge}, \cite{pycox}, \cite{tkott}, this 
conjecture was known to hold except possibly for type $E_8$. The algorithms 
developed in this paper allow us to verify Kottwitz's conjecture for type 
$E_8$ in a straightforward way (by an almost automatic procedure). Hence, 
this conjecture is now known to hold for any finite Coxeter group. 

Systematic experiments based on the methods presented in this paper suggest 
a general characterisation of left cells in terms of Lusztig's 
$\ba$-function \cite[Chap.~5]{LuBook} and a slight variation of 
Vogan's generalised $\tau$-invariant \cite[\S 3]{voga}; see 
Conjecture~\ref{sigmatau}.

This paper is organised as follows. In Section~\ref{sec1}, we briefly
recall the basic definitions concerning the polynomials $P_{y,w}$ and
the cells of $W$. In Section~\ref{secspec}, we present results of Lusztig 
\cite[Chap.~5]{LuBook} which show that the two-sided cells of $W$ are in 
bijective correspondence with the so-called ''special representations'' of 
$W$. By \cite[5.27]{LuBook}, this correspondence gives rise to the 
definition of a numerical function $w\mapsto \ba(w)$ on $W$, which is 
constant on the two-sided cells. In Section~\ref{seclead}, we discuss the 
problem of explicitly computing $\ba(w)$ for any given $w\in W$. The main 
idea is to use Lusztig's ''leading coefficients of character values'' 
\cite{LuBook}, \cite{Lu4} and their refinements for matrix representations 
introduced in \cite{my02}. In Section~\ref{sec2}, we recall some basic 
results about the Kazhdan--Lusztig star operations \cite{KL}; these lead 
to the definition of Vogan's \cite{voga} ''generalized $\tau$-invariant''
for elements of $W$. We then show how to determine a (relatively small) 
set $\fCla(W)$ of left cells of $W$ such that {\em any} left cell of $W$ 
can be reached from a unique cell in $\fCla(W)$ by a straightforward 
procedure (repeated applications of star operations); see 
Remark~\ref{defrel1}. In Section~\ref{sece8}, we apply these general 
methods to type $E_8$. Here, there are $101796$ left cells in total but 
the set $\fCla(W)$ contains only $106$ left cells. Using the knowledge 
of $\fCla(W)$, we obtain the desired algorithms for efficiently dealing 
with all the left cells of $W$. Applications, including Kottwitz' 
conjecture, are discussed in the final Section~\ref{secappl}.

%

A key idea in this work is to use a relatively small subset $\breve{\cD}
\subseteq W$ with the following properties: (1) It is defined in general 
terms and is known to contain---by a theoretical argument, see 
Proposition~\ref{pyc56}---representatives of all left cells of $W$, (2)  
there is a general algorithm for the determination of $\breve{\cD}$ (see 
\cite[\S 5]{pycox}) and (3) this algorithm also produces additional
information like values of the $\ba$-function and the characters of 
cell representations. It can be shown that $\breve{\cD}$ is in fact the set 
of ''distinguished involutions'' as defined by Lusztig \cite{Lu2}, but we do 
not need this result here. In our approach, knowing $\breve{\cD}$ constitutes 
the first step in describing the left cells of $W$. In type $E_8$, for
example, it quickly leads to a new and independent proof of the main result 
of Chen \cite{chene8}; see Example~\ref{tauE8}. This is the basis for
the experiments leading to the formulation of Conjecture~\ref{sigmatau}.

\section{Kazhdan--Lusztig polynomials and cells} \label{sec1}

Let $W$ be a Coxeter group with generating set $S$. Let $l\colon W 
\rightarrow \Z_{\geq 0}$ be the usual length function with respect to $S$.
We briefly recall the definition of cells from \cite{KL}. For this purpose,
let $\cH$ be the one-parameter generic Iwahori--Hecke algebra. This is 
an associative algebra over the ring $A=\Z[v,v^{-1}]$ of Laurent polynomials 
in one variable~$v$. The algebra $\cH$ is free as an $A$-module with basis 
$\{T_w\mid w\in W\}$. The multiplication is determined by:
\[ T_sT_w=\left\{\begin{array}{cl} T_{sw} & \quad \mbox{if $l(sw)>l(w)$},\\
T_{sw}+(v-v^{-1})T_w & \quad \mbox{if $l(sw)<l(w)$},\end{array}\right.\]
where $s\in S$ and $w\in W$. For basic properties of $W$ and $\cH$, we
refer to \cite{gepf}, \cite{Lusztig03}.  Let $\{C_w'\mid w\in W\}$ be the
new basis of $\cH$ introduced in \cite[Theorem~1.1]{KL}. The change of basis
to the old basis is given by equations
\[ C_w'=\sum_{y\in W} v^{l(y)-l(w)}P_{y,w} T_y\qquad \mbox{where}\qquad
P_{y,w}\in \Z[v].\]
Here are some of the properties of the polynomials $P_{y,w}$ (see also 
Remark~\ref{rem1} below). We have $P_{w,w}=1$ and $P_{y,w}=0$ unless $y\leq w$, 
where $\leq$ is the Bruhat--Chevalley order. If $y<w$, then $P_{y,w}\in\Z[v]$ 
has constant term $1$ and degree at most $l(w)-l(y)-1$; furthermore, only
even powers of $v$ will occur in $P_{y,w}$. If $y<w$, we will denote by 
$\mu(y,w)$ the coefficient of $v^{l(w)-l(y)-1}$ in $P_{y,w}$. (Thus, 
$\mu(y,w)=0$ if $l(y) \equiv l(w)\bmod 2$.) We write $y \leftarrow_{L} w$
if one of the following conditions holds:
\begin{itemize}
\item $y=w$ or 
\item $y<w$ and $\mu(y,w)\neq 0$ or
\item there exists some $s\in S$ such that $y=sw>w$.
\end{itemize}
The Kazhdan--Lusztig left pre-order $\leq_{L}$ is the transitive closure
of the relation $\leftarrow_{L}$, that is, we have $x\leq_{L} y$ if there 
exists a sequence $x=x_0,x_1, \ldots,x_k=y$ of elements in $W$ such that 
$x_{i-1} \leftarrow_{L} x_i$ for all~$i$. The equivalence relation associated
with $\leq_{L}$ will be denoted by $\sim_{L}$ and the corresponding
equivalence classes are called the {\em left cells} of $W$.

We write $x \leq_{R} y$ if $x^{-1} \leq_{L} y^{-1}$. The
equivalence relation associated with $\leq_{R}$ will be denoted by
$\sim_{R}$ and the corresponding equivalence classes are called the
{\em right cells} of $W$. Finally, we define a pre-order $\leq_{LR}$
by the condition that $x\leq_{LR} y$ if there exists a sequence $x=x_0,
x_1,\ldots, x_k=y$ such that, for each $i \in \{1,\ldots,k\}$, we have 
$x_{i-1} \leq_{L} x_i$ or $x_{i-1}\leq_{R} x_i$. The equivalence 
relation associated with $\leq_{LR}$ will be denoted by $\sim_{LR}$ 
and the corresponding equivalence classes are called the {\em two-sided 
cells} of $W$.

\begin{rem} \label{rem0b} Among other reasons, the partition of $W$ into 
left cells is important because it gives rise to representations of $W$ 
and the corresponding algebra $\cH$ in terms of so-called ''$W$-graphs''
(see \cite[Theorem~1.3]{KL}). Indeed, let $\Gamma$ be a left cell and 
$[\Gamma]_A$ be an $A$-module which is free over $A$ with basis $\{e_x
\mid x\in \Gamma\}$. Then $[\Gamma]_A$ is an $\cH$-module, where the 
action of $T_s$ ($s\in S$) is given as follows.
\[T_s.e_y=\left\{\begin{array}{cl} -v^{-1} e_y & \quad \mbox{if $sy<y$},\\
ve_y+\displaystyle \sum_{x\in \Gamma\,:\, sx<x} \tilde{\mu}(x,y) e_x & \quad 
\mbox{if $sy>y$}.\end{array}\right.\]
Here, we set $\tilde{\mu}(x,y)=\mu(x,y)$ if $x<y$ and $\tilde{\mu}(x,y)=
\mu(y,x)$ if $y<x$; otherwise, we set $\tilde{\mu}(x,y)=0$. (See also 
\cite[Def.~1.2]{KL}.)
\end{rem}


\begin{rem} \label{rem1} By the definitions, the cells of $W$ are determined 
via the knowledge of the coefficients $\mu(y,w)$ of the polynomials
$P_{y,w}$. Now, there is a recursive and purely combinatorial algorithm for 
the computation of $P_{y,w}$. Assume that $w \neq 1$ and let $s \in S$ be 
such that $sw <w$. Then we have (see \cite[(2.2.c) and (2.3.g)]{KL}):
\begin{itemize}
\item If $sy>y$, then $P_{y,w}=P_{sy,w}$.
\item If $sy<y$, then 
$\displaystyle P_{y,w}=P_{sy,sw}+v^2P_{y,sw}- \sum_{y \leq z <sw,\,sz<z} 
\mu(z,sw)v^{l(w)-l(z)}P_{y,z}$.
\end{itemize}
For large $W$, these formulae quickly become unusable. Using the concept 
of ''induction of cells''\cite{myind}, one can improve these recursion 
formulae, as explained in \cite[\S 4]{pycox}. 
In this way, we obtain an algorithm for the computation of left cells which 
works quite efficiently for all finite $W$ of rank $\leq 7$; see 
\cite[Table~2, p.~246]{pycox}. For example, let $W$ be of type $E_6$. The 
left cells in this case have been determined by Tong \cite{tong}. Using 
the computer algebra package {\sf PyCox} \cite{pyc14}, we obtain the cells
by an automatic procedure as follows. 
\begin{verbatim}
    >>> W=coxeter("E", 6); W.order
    51840
    >>> kl=klcells(W, 1, v)
    #I 652 left cells (21 non-equivalent), mues:1
\end{verbatim}
(This takes about $30$ seconds; of course, this varies with the available 
computer. See the help menu of {\tt klcells} for further information 
about the format of this command.) This also works for type $E_7$; the 
computation of the $6364$ left cells using {\tt klcells} takes about 
$4$ hours in this case. (See Chen--Shi \cite{chenshi} where different 
methods were used.)---However, type $E_8$ remains by far out of reach in 
this approach. 
\end{rem} 

\begin{rem} \label{rem0c} Assume that $W$ is finite. Then we denote by
$\Irr(W)$ the set of simple $\R[W]$-modules (up to isomorphism). Note that 
$\R$ is a splitting field for $W$ (see \cite[6.3.8]{gepf}). Let $\Gamma$ be 
a left cell. By specializing $v\rightarrow 1$, the $\cH$-module $[\Gamma]_A$ 
in Remark~\ref{rem0b} becomes an $\R[W]$-module which we denote by 
$[\Gamma]_1$. For any $E\in\Irr(W)$, we denote by $m(\Gamma,E)$ the 
multiplicity of $E$ as a constituent of $[\Gamma]_1$. We shall also need 
to address the problem of computing these multiplicities without having 
to work out the character values of $[\Gamma]_1$; this will be done in
Proposition~\ref{pyc56} below. 
\end{rem}

\section{Two-sided cells and special representations} \label{secspec}

We shall assume from now on that $W$ is finite. By the above definitions, 
the two-sided cells of $W$ are derived from the knowledge of the relation
$\leq_L$. Thus, it would seem that the determination of the two-sided cells 
is at least as difficult as the determination of the left cells. There is, 
however, a different way to approach the two-sided cells, using the 
representation theory of $W$. This is based on the following constructions. 
Since the left cells form a partition of $W$, we have a direct sum
decomposition of left $\R[W]$-modules
\[ \R[W] \cong \bigoplus_{\text{$\Gamma$ left cell of $W$}} [\Gamma]_1.\]
Hence, given $E \in \Irr(W)$, there exists a left cell $\Gamma$ such that 
$E$ is a constituent of $[\Gamma]_1$. 

\begin{defn} \label{def0} Let $E\in \Irr(W)$. Then all left 
cells $\Gamma$ such that $E$ is a constituent of $[\Gamma]_1$ are 
contained in the same two-sided cell. (See \cite[5.1, 5.15]{LuBook} or 
\cite[\S 2.2]{geja}.) This two-sided cell, therefore, only depends on 
$E$ and will be denoted by $\cF_E$. Thus, we obtain a partition
\[\Irr(W)=\bigsqcup_{\cF \text{ two-sided cell}} \Irr(W\mid \cF),\]
where $\Irr(W\mid \cF)$ consists of all $E \in \Irr(W)$ such that
$\cF_E=\cF$. 
\end{defn}

It is remarkable that one can prove some things about the above partition 
of $\Irr(W)$ without first working out the two-sided cells. To state this 
more precisely, we need some further notation. Let $\cH$ be the one-parameter
generic Iwahori--Hecke algebra associated with $W$, as in the previous
section. Let $K=\R(v)$. By extension of scalars, we obtain a $K$-algebra 
$\cH_K=K \otimes_A \cH$. It is known that $\cH_K$ is split semisimple 
and abstractly isomorphic to $K[W]$ (see \cite[9.3.5]{gepf}); furthermore, 
the map $v\mapsto 1$ induces a bijection between $\Irr(\cH_K)$ and 
$\Irr(W)$ (see \cite[8.1.7]{gepf}). Given $E \in \Irr(W)$, we denote by 
$E_v$ the corresponding irreducible representation of $\cH_K$. We have 
$\mbox{trace}(T_w,E_v)\in \R[v,v^{-1}]$ for all $w \in W$ (see 
\cite[9.3.5]{gepf}). We define
\[ \ba_E:=\min\{i\in \Z_{\geq 0}\mid v^i\, \mbox{trace}(T_w,E_v) \in \R[v] 
\mbox{ for all $w \in W$}\}.\]
Finally, let $\bb_E$ be the smallest integer $i\geq 0$ such that $E$ 
occurs as a constituent of the $i$-th symmetric power of the natural 
reflection representation of $W$. Then, as in \cite[4.1]{LuBook}, the set 
of {\em special representations} of $W$ is defined by   
\[\cS(W):=\{E\in \Irr(W) \mid \ba_E=\bb_E\}.\]
Now we can state:

\begin{thm}[Lusztig] \label{canrep} Let $\cF$ be a two-sided cell of $W$.
Then $\cF=\cF_{E_0}$ for a unique $E_0 \in\cS(W)$. Furthermore, the 
function $E\mapsto \ba_E$ is constant on $\Irr(W\mid \cF)$.
\end{thm}

The functions $E\mapsto \ba_E$, $E\mapsto \bb_E$ and, hence, the sets
$\cS(W)$ are explicitly known in all cases; see the tables in 
\cite[Chap.~4]{LuBook} (for finite Weyl groups) and \cite[\S 6.5]{gepf} 
(where the types $I_2(m)$, $H_3$, $H_4$ are included in the discussion).

\begin{table}[htbp] 
\caption{The $46$ special representations for $W$ of type $E_8$} 
\label{fame8} 
\begin{center}
$\begin{array}{cc} \hline E & \ba_E \\\hline
1_x& 0\\
8_z &1\\
35_x& 2 \\
112_z &3\\ 
210_x& 4\\ 
560_z&5\\ 
567_x&6\\ 
700_x&6\\ 
1400_z &7\\ 
1400_x& 8\\ \hline
\end{array}\qquad
\begin{array}{cc} \hline E & \ba_E \\\hline
3240_z&9\\ 
2268_x&10\\
2240_x&10\\ 
4096_z&11\\
525_x&12\\
4200_x&12 \\ 
2800_z&13\\ 
4536_z&13 \\ 
2835_x&14\\ 
6075_x&14 \\\hline
\end{array}\qquad
\begin{array}{cc} \hline E & \ba_E \\\hline
4200_z&15\\ 
5600_z&15 \\ 
4480_y&16\\
2100_y&20\\ 
4200_z'&21\\ 
5600_z'&21\\
2835_x'&22\\
6075_x'&22\\ 
4536_z'&23\\ 
4200_x'&24\\ \hline
\end{array}\qquad
\begin{array}{cc} \hline E & \ba_E \\\hline
2800_z'&25\\ 
4096_x' &26 \\
2240_x'&28\\ 
2268_x'& 30 \\
3240_z'&31\\ 
1400_x'& 32\\ 
525_x'&36\\
1400_z'& 37\\
700_x'&42 \\
567_x'&46\\ \hline
\end{array}\qquad
\begin{array}{cc} \hline E & \ba_E \\\hline
560_z'&47\\
210_x'& 52\\ 
112_z' &63\\
35_x'&74\\ 
8_z'& 91\\ 
1_x' & 120\\\hline \\\\\\\\
\end{array}$
\end{center}
\end{table}

For example, if $W$ is of type $E_8$, we have $|\Irr(W)|=112$ and there 
are $46$ special representations; they are listed in Table~\ref{fame8} 
(which is taken from \cite[4.13.1]{LuBook}). Consequently, by 
Theorem~\ref{canrep}, we already know that there are $46$ two-sided cells 
of $W$. Using \cite[5.25.2, 5.26, 12.3.7]{LuBook}, it also follows that
\[\mbox{Number of left cells of $W$ } =\sum_{E_0\in \cS(W)} \dim E_0 =
101796.\]
We shall not need the latter result, since we will obtain this number 
independently in the course of our computations; see Example~\ref{tauE8}.

\begin{defn}[Lusztig \protect{\cite[5.27]{LuBook}}] \label{aval} Let
$w\in W$. If $\cF$ is the two-sided cell containing $w$, then we set 
$\ba(w):=\ba_E$ where $E\in\Irr(W\mid\cF)$. (By Theorem~\ref{canrep}, 
this is well-defined.) 
\end{defn}

{\em Comments on the proof of Theorem~\ref{canrep}}. Contrary to the 
previous results, and those discussed in Section~\ref{seclead} below,
there does not seem to exist an elementary direct proof of 
Theorem~\ref{canrep}. Using standard operations in the character ring 
of $W$ (induction from parabolic subgroups, tensoring with the sign 
character), Lusztig \cite[4.2]{LuBook} has defined another partition 
of $\Irr(W)$ into so-called ``families''. As shown in 
\cite[Theorem~5.25]{LuBook} (for finite Weyl groups, using deep results 
from the theory of Lie algebras and algebraic groups), we have 
$\cF_E=\cF_{E'}$ if and only if $E,E'$ belong to the same ''family''; a
relatively simple direct argument for the types $I_2(m)$, $H_3$, $H_4$ can be
found in \cite[Example~3.6]{klord}. (Another proof, for general $W$, is given
in \cite[Prop.~23.3]{Lusztig03}, which relies on certain positivity properties 
for $\cH$. There is now a general, purely algebraic proof of these 
positivity properties, due to Elias--Williamson \cite[Cor.~1.2]{EW}.) 
Then the statements in Theorem~\ref{canrep} follow from the analogous 
results for ''families'', and these are contained in 
\cite[Chap.~4]{LuBook} (for finite Weyl groups) and \cite[\S 6.5]{gepf} 
(where the types $I_2(m)$, $H_3$, $H_4$ are included in the discussion).\qed

\begin{rem} \label{rem0} By the definitions, it is clear that every
two-sided cell is at the same time a union of left cells and a union of
right cells. It is, however, not clear at all that two-sided cells are the 
smallest subsets of $W$ with this property. This follows from the following 
implication, where $x,y\in W$:
\begin{equation*} 
x \leq_L y \qquad \mbox{and}\qquad x\sim_{LR} y \qquad \Rightarrow \qquad
x\sim_L y.\tag{A}
\end{equation*}
The fact that (A) holds is implicit in the proof of Theorem~\ref{canrep}.
When $W$ is a finite Weyl group, (A) was first proved by Lusztig; see 
\cite[Lemma~4.1]{Lu0} or \cite[Cor.~5.5]{Lu1} (using similar methods as 
discussed in the proof of Theorem~\ref{canrep}). For type $H_4$, (A) has 
been verified in \cite[Cor.~3.3]{Al}; a similar verification also works 
for $H_3$. For type $I_2(m)$, see \cite[8.7]{Lusztig03}. Once (A) is known 
to hold, it follows that the partition of $W$ into left cells determines 
the partitions into right cells and into two-sided cells. 
\end{rem}


\section{Cells and leading coefficients} \label{seclead}

In the previous section, we have seen that one can attach to any element
$w\in W$ a numerical value $\ba(w)$: we have $\ba(w)=\ba_{E}$ where
$E\in\Irr(W)$ is such that $w\in \cF_E$. The purpose of this section is to
address the following problem:
\begin{center}
{\em Given $w\in W$, how can we actually find some $E\in\Irr(W)$ such that 
$w\in \cF_E$~?}
\end{center}
We will see that this problem can be solved using Lusztig's {\em leading 
coefficients of character values} \cite{LuBook}, \cite{Lu4} and their 
refinements introduced in \cite{my02}. Let $E\in\Irr(W)$ and consider
the corresponding irreducible representation $E_v$ of $\cH_K$. Recall that
\[v^{\ba_E}\,\mbox{trace}(T_w,E_v) \in \R[v]\qquad \mbox{for all $w\in W$}.\]
Consequently, there are unique numbers $c_{w,E} \in \R$ ($w\in W$) such that
\[v^{\ba_E}\,\mbox{trace}(T_w,E_v) =(-1)^{l(w)}\,c_{w,E}+\mbox{``higher 
terms''},\]
where ``higher terms'' means an $\R$-linear combination of monomials
$v^i$ where $i>0$. Given $E$, there is at least one $w \in W$ such that 
$c_{w,E}\neq 0$ (by the definition of $\ba_E$). Since $\mbox{trace}(T_w,
E_v)= \mbox{trace}(T_{w^{-1}},E_v)$ for all $w \in W$ (see 
\cite[8.2.6]{gepf}), we certainly have
\[ c_{w,E}=c_{w^{-1},E} \qquad \mbox{for all $w \in W$}.\]
These numbers are called the {\em leading coefficients of character values};
see \cite{LuBook}, \cite{Lu4}. We now describe some general constructions
involving these leading coefficients. 

Let $E\in\Irr(W)$. Since the sum of all $c_{w,E}^2$ ($w\in W$) is strictly 
positive, we can write that sum as $f_E\,\dim E$ where $f_E\in \R_{>0}$. 
With this notation, we can state:

\begin{prop}[Cf.\ \protect{\cite[5.6]{pycox}}] \label{pyc56} We define
real numbers as follows:
\[ \breve{n}_w:=\sum_{E\in\Irr(W)} f_E^{-1}c_{w,E} \qquad \mbox{for all
$w\in W$}.\]
Then the following hold.
\begin{itemize}
\item[(a)] Let $\Gamma$ be a left cell and $E\in\Irr(W)$. Then the 
multiplicity $m(\Gamma,E)$ (see Remark~\ref{rem0c}) is given by
\[ m(\Gamma,E)=\sum_{w\in \Gamma} \breve{n}_wc_{w,E}.\]
\item[(b)] Every left cell of $W$ contains an element of the set 
$\breve{\cD}:=\{w \in W\mid \breve{n}_w\neq 0\}$. Thus, the number of 
left cells of $W$ is less than or equal to $|\breve{\cD}|$.
\end{itemize}
\end{prop}

\begin{proof} (a) Let $E,E'\in\Irr(W)$. By \cite[5.8]{LuBook} (for finite 
Weyl groups) and \cite[Cor.~3.8]{myfrob} (in general), we have the 
following orthogonality relations:
\[ \sum_{w\in \Gamma} c_{w,E}c_{w,E'}=\left\{\begin{array}{cl}
m(\Gamma,E)f_E & \quad \mbox{if $E\cong E'$},\\ 0 & \quad \mbox{otherwise}.
\end{array}\right.\]
Using these relations and the defining formula for $\breve{n}_w$, the 
expression $\sum_{w\in \Gamma}\breve{n}_wc_{w,E}$ reduces to $m(\Gamma,E)$.
(See also \cite[Example~1.8.5]{geja}.) 

(b) Let $\Gamma$ be a left cell. Then there is some $E\in\Irr(W)$ such 
that $m(\Gamma,E)>0$. So (a) shows that there exists some $w\in \Gamma$ 
such that $\breve{n}_w\neq 0$, that is, we have $w\in\Gamma\cap\breve{\cD}$.
\end{proof}

\begin{cor}[Cf.\ Lusztig \protect{\cite[Lemma~5.2]{LuBook}}] \label{prop1} 
Let $w\in W$ and assume that there exists some $E\in \Irr(W)$ such that 
$c_{w,E}\neq 0$. Then $w\in\cF_E$.
\end{cor}

\begin{proof} Let $\Gamma$ be the left cell such that $w\in\Gamma$. Then 
the orthogonality relations used in the above proof show that $\sum_{x\in
\Gamma} c_{x,E}^2=f_Em(\Gamma,E)$. Since the sum contains some strictly 
positive term (the one corresponding to $x=w$), we have $m(\Gamma,E)>0$ 
and so $\Gamma\subseteq \cF_E$. 
\end{proof}

The above methods allow us to solve the problem raised at the beginning
of this section for those elements $w\in W$ for which there exists
some $E\in\Irr(W)$ such that $c_{w,E}\neq 0$. The new ingredient to deal 
with {\em all} elements of $W$ are the leading matrix coefficients 
introduced in \cite{my02}. For each $E\in \Irr(W)$, we consider a matrix 
representation
\[ \rho_v^E\colon \cH_K\rightarrow M_{d_E}(K)\qquad (d_E=\dim E)\]
affording $E_v\in\Irr(\cH_K)$. Now recall that $K$ is the field of 
fractions of $\R[v]$. Let  us consider the discrete valuation ring 
\[ \cO=\{f/g\in K\mid f,g\in\R[v], g(0)\neq 0\}\subseteq K.\]
In particular, we have a ring homomorphism $\cO\rightarrow \R$ given
by evaluation at $0$. Following \cite[2.2, 4.3]{myedin}, we say that
$\rho_v^E$ is a {\em balanced representation} if there exists a
symmetric matrix $\Omega^\lambda \in M_{d_E}(\cO)$ such that 
\[\det(\Omega^E)\in\cO^\times \qquad \mbox{and} \qquad \Omega^E \,
\rho^E_v (T_{w^{-1}})=\rho_v^E(T_w)^{\operatorname{tr}} 
\,\Omega^E \quad \mbox{for all $w \in W$}.\]
By \cite[2.3]{myedin} (see also \cite[\S 1.4]{geja}), we can always choose 
$\rho_v^E$ to be balanced;  let us now assume that this is the case. Then 
we have  
\[v^{\ba_E}\rho_v^E(T_w) \in M_{d_E}(\cO)\qquad\mbox{for all $w\in W$}.\] 
For $1\leq i,j\leq d_E$, we denote by $c_{w,E}^{ij}\in \R$ the value
at $0$ of the $(i,j)$-entry of the matrix $(-1)^{l(w)}v^{\ba_E}\rho_v^E
(T_w)$. These real numbers are called {\em leading matrix coefficients}. 
Note that
\[ c_{w,E}=\sum_{1\leq i \leq d_E} c_{w,E}^{ii} \qquad \mbox{for all
$w\in W$},\]
where $c_{w,E}$ are Lusztig's leading coefficients of character values as
introduced earlier.

\begin{defn}[Cf.\ \protect{\cite[Def.~3.1]{p115}, \cite[\S 1.6]{geja}}] 
\label{deflmc} Let $w\in W$ and $E\in \Irr(W)$. Then write 
$E\leftrightsquigarrow w$ if $c_{w,E}^{ij}\neq 0$ for some $i,j\in 
\{1,\ldots,d_E\}$. By \cite[3.9]{myedin}, the relation 
$\leftrightsquigarrow$ does not depend on the choice of the balanced 
representation $\rho_v^E$.
\end{defn}

\begin{lem} \label{prop2} Let $w\in W$. Then there exists some 
$E\in \Irr(W)$ such that $E\leftrightsquigarrow w$. For any
such $E$, we have $w\in \cF_E$. Consequently, we have $\ba(w)=\ba_E$
(see Definition~\ref{aval}).
\end{lem}

\begin{proof} By the orthogonality relations in \cite[3.3]{myedin},
we have
\[1=\sum_{E\in\Irr(W)} \sum_{1\leq i,j\leq d_E} c_{w,E}^{ij}
c_{w^{-1},E}^{ji}.\]
Hence, there exists some $E\in\Irr(W)$ such that $E\leftrightsquigarrow w$.
Now let $E\in\Irr(W)$ be arbitrary such that $E\leftrightsquigarrow w$.
Let $\Gamma$ be the left cell such that $w\in\Gamma$. Then, by 
\cite[Prop.~4.7]{my02}, \cite[Lemma~3.2]{p115}, we have $m(\Gamma,E)>0$ 
and so $w\in\cF_E$.
\end{proof}

\begin{rem} \label{remHY} Let $W$ be of exceptional type. Then Lusztig 
\cite[\S 5]{Lu0} (type $H_3$), Alvis--Lusztig \cite{AlLu} ($H_4$), Naruse 
\cite{Naruse0} ($F_4$, $E_6$), Howlett--Yin \cite{HowYin} ($E_6$, $E_7$) 
and Howlett \cite{How} ($E_8$) explicitly determined matrix representations 
$\rho_v^E$ for all $E\in\Irr(W)$ (in terms of abstract $W$-graphs, as 
defined in \cite{KL}). Thus, for each $E\in\Irr(W)$, we are given explicit 
matrices $\{\rho_v^E(T_s) \mid s\in S\}$, with entries in $\R[v,v^{-1}]$. 
In \cite[Example~4.6]{myedin} (types $H_3$, $H_4$) and 
\cite[\S 4]{gemu} (types $F_4$, $E_6$. $E_7$, $E_8$), it is shown how to 
construct symmetric matrices $\Omega^E\in M_{d_E}(\R[v])$ such that some 
entry of $\Omega^E$ has a non-zero constant term and such that 
\[\Omega^E \,\rho^E_v (T_{s})=\rho_v^E(T_s)^{\operatorname{tr}} 
\,\Omega^E \qquad \mbox{for all $s\in S$}.\]
One can check that, in all cases, $\det(\Omega^E)\in\R[v]$ has a non-zero
constant term. Hence, the above matrix representations are balanced in the 
sense defined above. (It is conjectured in \cite[4.6]{myedin} that every
$W$-graph representation is automatically balanced; see also 
\cite[1.4.14]{geja}.) In Michel's development version of the computer
algebra package {\sf CHEVIE} \cite{mich}, these representations are obtained 
via the command {\tt Representations}. Thus, given an element $w\in W$, we 
consider a reduced expression $w=s_1s_2\cdots s_{l}$ with $s_i\in S$ 
and obtain $\rho_v^E(T_w)=\rho_v^E(T_{s_1})\cdots \rho_v^E(T_{s_l})$. 
We can then check if the relation $E\leftrightsquigarrow w$ holds.---Note
that, in type $E_8$, we have $\max\{l(w)\}=120$ and $\max\{\dim E\}=7168$;
in extreme cases, the computation of the matrix $\rho_v^E(T_w)$ may take 
$10$ minutes or even more. Therefore, we will have to try to reduce as much 
as possible the number of elements $w\in W$ for which we need to apply this
method to work out $\ba(w)$. The techniques for achieving such a reduction 
will be discussed in the following section.

(The computation \cite{gemu} of the symmetric matrices $\Omega^E$ for
type $E_8$ took several months; these matrices are available upon request. 
Note that we do not need to know them for establishing the relation 
$E\leftrightsquigarrow w$; as far as type $E_8$ is concerned, they were
{\em only} needed to prove that Howlett's matrix representations indeed 
are balanced.)
\end{rem}

\section{Star operations and induction of cells} \label{sec2}

We now turn to the determination of the left cells of $W$. A first 
approximation is obtained by the following constructions, which already 
appeared in \cite{KL}. For any $w\in W$, we denote by $\cR(w):=\{s\in 
S\mid ws<w\}$ the {\em right descent set} of $w$.

\begin{prop}[\protect{\cite[Prop.~2.4]{KL}}] \label{klright}
Let $x,y\in W$. If $x\sim_{L} y$, then $\cR(x)=\cR(y)$. Thus, for any 
$I\subseteq  S$, the set $\{w\in W\mid \cR(w)=I\}$ is a union of
left cells of $W$.
\end{prop}

The above result can be refined considerably. Following Kazhdan--Lusztig 
\cite[4.1]{KL}, let us consider two generators $s,t\in S$ such that $st$ 
has order~$3$. We set 
\[ \cD_R(s,t):=\{w\in W\mid \cR(w)\cap \{s,t\} \mbox{ has exactly one
element}\}.\]
If $w\in \cD_R(s,t)$, then exactly one of the elements $ws,wt$ belongs to
$\cD_R(s,t)$; we denote it $w^*$. The map
\[ \sigma_{s,t}\colon \cD_R(s,t) \rightarrow \cD_R(s,t), \qquad  
w\mapsto  w^*,\]
is an involution, called {\em star operation}. Now let $\Gamma$ be a left 
cell. By Proposition~\ref{klright}, we have $\Gamma\subseteq \cD_R(s,t)$ or 
$\Gamma \cap \cD_R(s,t)=\varnothing$.

\begin{prop}[\protect{\cite[Cor.~4.3]{KL}}] \label{klstar}
Let $s,t\in S$ be as above and $\Gamma$ be a left cell such that
$\Gamma\subseteq \cD_R(s,t)$. Then $\Gamma^*:=\{w^*\mid w\in \Gamma\}$ 
also is a left cell and the map $[\Gamma]_A \rightarrow [\Gamma^*]_A$, 
$e_w\mapsto e_{w^*}$, is an isomorphism of $\cH$-modules. Furthermore,
$w\sim_{R} w^*$ for all $w\in \Gamma$.
\end{prop}

\begin{rem} \label{leftkl} Let $y,w\in W$. We write 
$y\stackrel{*}{\leftrightarrow} w$ if $w$ is obtained from $y$ via repeated
applications of star operations, that is, there is a sequence of elements 
$y=y_0,y_1,\ldots, y_m=w$ in $W$ such that, for each $i\in\{1,\ldots,m\}$, 
we have $y_i=\sigma_{s_i,t_i}(y_{i-1})$ for some $s_i,t_i\in S$ such that 
$s_it_i$ has order $3$ and $y_{i-1}\in \cD_R(s_i,t_i)$. The set 
\[ R^*(w):=\{y\in W \mid y\stackrel{*}{\leftrightarrow} w\}\]
will be called the {\em (right) star orbit} of~$w$. By 
Proposition~\ref{klstar}, $R^*(w)$ is contained in a right cell of $W$. 
Similarly, the {\em left star orbit} of $w$ is defined by 
\[L^*(w):= \{y \in W\mid y^{-1}\stackrel{*}{\leftrightarrow} w^{-1}\}.\]
Since $y\sim_L w\Leftrightarrow y^{-1}\sim_R w^{-1}$, the set $L^*(w)$ 
is contained in a left cell. In {\sf PyCox}, these operations are 
performed by the commands {\tt klstarorbitelm} and {\tt leftklstarorbitelm}.
\end{rem}

\begin{rem} \label{klstar1} Let $\fCl(W)$ be the set of all left cells of 
$W$. The above result shows that the star operations permute the elements
of $\fCl(W)$. We shall denote by $\fCla(W)\subseteq \fCl(W)$ a subset 
such that any $\Gamma\in \fCl(W)$ can be reached from a unique element
of $\fCla(W)$ via repeated applications of star operations. From a practical 
point of view, this is particularly useful for storing results on a 
computer: Since star operations are performed quite easily and efficiently,
it will only be necessary to store the cells in $\fCla(W)$. The {\sf PyCox} 
command {\tt klcells} returns such a subset $\fCla(W)$ of the set of all 
left cells of $W$. 

For example, if $W$ is of type $E_6$, then the output 
of the {\tt klcells} command in Remark~\ref{rem1} tells us that $|\fCl(W)|=
652$ and $|\fCla(W)|=21$. 
\end{rem}

In combination, Propositions~\ref{klright} and \ref{klstar} show that, 
if $\Gamma$ is a left cell, then we not only have $\cR(y)=\cR(w)$ but also
$\cR(y^*)=\cR(w^*)$ for all $y,w\in\Gamma$. Iterating this process leads 
us to the following notion.

\begin{defn}[Cf.\ Vogan \protect{\cite[3.10]{voga}}, simply-laced 
version\footnote{The general version would also take into account 
generators $s,t\in S$ such that $st$ has order strictly bigger than~$3$; 
see \cite[3.10]{voga} and Remark~\ref{string} below, but the relations 
(and the proofs) are somewhat more complicated; the ''simply-laced''
version is sufficient for our purposes here.}] \label{deftau} Let $n\geq 0$ 
and $y,w\in W$. We define a relation $y \approx_n w$ inductively as follows. 
First, let $n=0$. Then $y\approx_0 w$ if $\cR(y)=\cR(w)$. Now let $n>0$ and 
assume that $y'\approx_{n-1} w'$ has been already defined for all
$y', w'\in W$. Then $y\approx_n w$ if $y\approx_{n-1} w$ and if 
$\sigma_{s,t}(y) \approx_{n-1} \sigma_{s,t}(w)$ for any $s,t\in S$ such
that $st$ has order $3$ and $y,w\in \cD_R(s,t)$.

If $y\approx_n w$ for all $n\geq 0$, then we say that $y,w$ have
the same {\em generalized $\tau$-invariant}. This defines an equivalence
relation on $W$; the corresponding equivalence classes will be called
$\tau$-cells. 

\end{defn}

\begin{cor} \label{cortau} Let $\Gamma$ be a left cell of $W$. Then
$\Gamma$ is contained in a $\tau$-cell.
\end{cor}

\begin{proof} This is clear by Propositions~\ref{klright} and~\ref{klstar}.
\end{proof}

Thus, by computing the $\tau$-cells of $W$, we obtain a 
first approximation to the partition of $W$ into left cells. This is 
quite powerful. For example, by Kazhdan--Lusztig \cite[\S 5]{KL}, the 
left cells in type $A_{n-1}$ (where $W$ is the symmetric group $\fS_n$) 
are precisely given by the $\tau$-cells. A similar result also holds for 
$W$ of type $E_6$, as shown by Tong \cite[\S 7]{tong}. In {\sf PyCox}, 
the $\tau$-cells are computed by the command {\tt gentaucells} (which just
implements the procedure described in Definition~\ref{deftau}; note that,
when $W$ is finite, one only needs to verify $y\approx_n w$ for a finite
number of $n$.) For example, Tong's result is recovered as follows.
\begin{verbatim}
    >>> W=coxeter("E", 6)
    >>> len(gentaucells(W, allwords(W)))
    652        # the number of left cells
\end{verbatim}

In general, the $\tau$-cells will just be unions of left cells.

Finally, we show how to obtain the complete set of all left cells of $W$
(and not just their representatives in $\breve{\cD}$ or $W^*$). This will 
rely on two techniques: (1) orbits under the star operations (see
Remark~\ref{klstar1}) and (2) the concept of ''induction of cells'' 
\cite{myind}. Let us explain how (2) works. Let $I\subseteq S$ and consider 
the parabolic subgroup $W_I\subseteq W$.  Then 
\[ X_I:=\{w\in W\mid \cR(w)\cap I=\varnothing\}\]
is the set of distinguished left coset representatives of $W_I$ in $W$. 
The map $X_I \times W_I \rightarrow W$, $(x,u) \mapsto xu$, is a bijection 
and we have $l(xu)=l(x)+l(u)$ for all $x\in X_I$ and $u\in W_I$; see 
\cite[\S 2.1]{gepf}. Thus, given $w\in W$, we can write uniquely $w=xu$
where $x\in X_I$ and $u\in W_I$. In this case, we denote $\prI_I(w):=u$.
Let $\sim_{L,I}$ be the equivalence relation on $W_I$ for which the 
equivalence classes are the left cells of $W_I$.

\begin{prop}[\protect{\cite{myind}}] \label{cellind} Let $I\subseteq S$.
If $w,w'\in W$ are such that $w\sim_L w'$, then $\prI_I(w)\sim_{L,I}
\prI_I(w')$. In particular, if $\Gamma$ is a left cell of $W_I$, then 
$X_I\Gamma$ is a union of left cells of $W$.
\end{prop}

%
We now show that the induction of cells is compatible with the star 
operations.

\begin{rem} \label{remapp} Let $I\subseteq S$ and $s,t\in I$ be generators 
such that $st$ has order $3$. Then we can define star operations with 
respect to $W_I$ and $s,t$.  As above we set 
\[ \cD_R^I(s,t):=\{u\in W_I\mid \cR(u)\cap \{s,t\} \mbox{ has exactly one
element}\}\]
and obtain a bijection $\sigma_{s,t}^I\colon \cD_R^I(s,t)\rightarrow
\cD_R^I(s,t)$, $u\mapsto u^*$. Now note that $\cD_R^I(s,t)=\cD_R(s,t) 
\cap W_I$ and  $\sigma_{s,t}^I$ is the restriction of $\sigma_{s,t}$ from 
$\cD_R(s,t)$ to $\cD_R^I(s,t)$. Further note that $X_I\cD_R^I(s,t)\subseteq 
\cD_R(s,t)$ and
\[ \sigma_{s,t}(xu)=x\sigma_{s,t}^I(u)\quad \mbox{for all $x\in X_I$ and 
$u\in \cD_R^I(s,t)$}.\]
Thus, more intuitively, we can also write $(xu)^*=xu^*$ for $x\in X_I$ and
$u\in \cD_R^I(s,t)$. 
\end{rem}

\begin{cor}[Cf.\ \protect{\cite[3.9, 3.10]{myrel}}] \label{cor1} In the 
setting of Remark~\ref{remapp}, let $\Gamma$ be a left cell of $W_I$ such 
that $\Gamma\subseteq \cD_R^I(s,t)$. Let $\Gamma^*:=\{u^* \mid w\in\Gamma\} 
\subseteq \cD_R^I(s,t)$ (where we use the star operation with respect to 
$W_I,s,t$). By Propositions~\ref{klstar} and \ref{cellind}, each of the 
two sets $X_I\Gamma$, $X_I\Gamma^*$ is a union of left cells of $W$. Now let 
$X_I\Gamma=\Gamma_1\amalg \ldots \amalg \Gamma_r$ be the partition of 
$X_I\Gamma$ into left cells of $W$. Then the star operation (with respect 
to $W,s,t$) is defined for each $\Gamma_i$ and $X_I\Gamma^*= \Gamma_1^* 
\amalg \ldots \amalg \Gamma_r^*$ is the partition of $X_I\Gamma^*$ into 
left cells of $W$. 
\end{cor}

\begin{proof} Since $X_I\Gamma\subseteq X_I\cD_R^I(s,t)\subseteq 
\cD_R(s,t)$, the star operation (with respect to $W,s,t$) is defined for
each element of $X_I\Gamma$. Now, the bijection $\Gamma
\rightarrow \Gamma^*$, $u\mapsto u^*$, induces a bijection 
\[ X_I\Gamma\rightarrow X_I\Gamma^*,\qquad xu\mapsto xu^* 
\quad (x\in X_I,u\in\Gamma).\]
By Remark~\ref{remapp}, we have $(xu)^*=xu^*$ for $x\in X_I$ and
$u\in \cD_R^I(s,t)$. Thus, the above bijection $X_I\Gamma\rightarrow
X_I\Gamma^*$ is the restriction of $\sigma_{s,t}$ from $\cD_R(s,t)$ to 
$X_I\Gamma$. Consequently, we have $\Gamma_i^*\subseteq X_I\Gamma^*$ for 
$1\leq i\leq r$ which yields the desired assertion. 
\end{proof}


\begin{rem} \label{defrel1} The practical use of Proposition~\ref{cellind} 
is as follows. Let $I\subsetneqq S$ and assume that the left cells of $W_I$ 
have already been determined; let $\fCl(W_I)=\{\Gamma_1,\ldots,\Gamma_r\}$. 
For each $i$, let $X_I\Gamma_i = Y_{i,1}\amalg \ldots \amalg Y_{i,t_i}$ be 
the decomposition into $\tau$-cells. Then the sets $\{ Y_{i,j} \mid 1\leq i 
\leq r,1\leq j\leq t_i\}$ form a partition of $W$; by Corollary~\ref{cortau} 
and Proposition~\ref{cellind}, each $Y_{i,j}$ is a union of left cells of $W$.

Using star operations, we can go one step further. Let $R\subseteq
\{1,\ldots,r\}$ be a subset such that $\fCla(W_I)=\{\Gamma_i\mid i\in R\}$
is a set of left cells as in Remark~\ref{klstar1}. Now let $\Gamma$ be
any left cell of $W$. Then Corollary~\ref{cor1} shows that there
exists a left cell $\Gamma'$ of $W$ which is contained in the set 
\[\Upsilon_I(W):=\bigcup_{i\in R,\,1\leq j\leq t_i} Y_{i,j} \;\subseteq W\]
and such that $\Gamma$ is obtained via repeated applications of star
operations from $\Gamma'$ (with respect to generators $s,t\in I$). Using
repeated applications of star operations with respect to {\em all}
generators in $S$, we obtain a subset $\fCla(W)\subseteq \Upsilon_I(W)$
as in Remark~\ref{klstar1}.
\end{rem}

\section{Type $E_8$} \label{sece8}

We shall now apply the methods developed so far to the group $W$ of 
type $E_8$. Recall that, by Lusztig's results in Section~\ref{secspec}, 
we already know that there are $46$ two-sided cells of $W$, corresponding 
to the special representations in $\cS(W)$. The first step is to consider 
the leading coefficients of character values from Section~\ref{seclead}.

\begin{exmp} \label{twoE8} Let $W$ be of type $E_8$. We have $|\Irr(W)|=112$.
The entries of the matrix of all leading coefficients $(c_{w,E})$ have 
been determined by Lusztig \cite[3.14]{Lu4}. This shows, for example,
that $|c_{w,E}|\leq 8$ for all $E,w$; however, the columns of the matrix 
are not matched with the various elements of $W$. Using \cite[\S 5, 
Algorithm~B]{pycox}, we can determine the matrix $(c_{w,E})$ together
with the labelling of the columns by the elements of~$W$. (See the 
description of the {\sf PyCox} command {\tt distinguishedinvolutions}; note 
that this computation takes nearly $18$ days and requires about $36$~GB of 
main memory.) By inspection of the output, we can verify the following 
statements.
\begin{itemize}
\item[(a)] The set $W^*:=\{w\in W\mid c_{w,E} \neq 0 \mbox{ for some $E\in 
\Irr(W)$}\}$ contains precisely $208422$ elements; furthermore, 
$|\breve{\cD}|= 101796$ and $\breve{n}_d=1$ for all $d\in\breve{\cD}$, 
\item[(b)] Let $\cI:=\{w\in W\mid w^2=1\}$ be the set of involutions in
$W$. This set $\cI$ is computed as explained in \cite[\S 2]{pycox}; we 
have $|\cI|=199952$ and $\breve{\cD}\subseteq \cI\subseteq W^*$. 
\item[(c)] For any $E_0\in\cS(W)$, we set $\cF_{E_0}^*:=\{w\in W\mid 
c_{w,E_0} \neq 0\}$. Then we have 
\[W^*:=\{w\in W\mid c_{w,E} \neq 0 \mbox{ for some $E\in \Irr(W)$}\}
=\bigcup_{E_0\in \cS(W)} \cF_{E_0}^*,\]
where the union on the right hand side is disjoint. 
\end{itemize}
Thus, for any $w\in W^*$, there is a unique $E_0\in\cS(W)$ such that
$c_{w,E_0}\neq 0$. Then we have $w\in\cF_{E_0}$ and so $\ba(w)=\ba_{E_0}$
(see Definition~\ref{aval}).
\end{exmp}

\begin{rem} \label{twoE8g} Most of the statements in Example~\ref{twoE8}
can also be deduced from theoretical arguments and, hence, hold in general.
(But we still need to know the sets $\cI$, $\breve{\cD}$ explicitly in type 
$E_8$ in the subsequent discussion.) For example, if $W$ is a finite Weyl 
group, then (c) holds by \cite[Prop.~7.1]{LuBook} and \cite[3.14]{Lu4}. If
$W$ is of type $H_3$, $H_4$ or $I_2(m)$, then (c) is checked by explicit 
computation; see \cite[Rem.~5.12]{pycox}. The inclusion $\cI\subseteq W^*$ 
in (b) holds by the remark just after \cite[3.5(b)]{Lu4}. The argument for
proving this works in general, once the positivity properties already 
mentioned in the comments on the proof of Theorem~\ref{canrep} are known to
hold.  Similarly, as observed in \cite[3.7]{p115}, these positivity 
properties together with the techniques in \cite{Lu4} also imply that 
$\breve{\cD}$ is the set of ''distinguished involutions'' defined by 
Lusztig \cite{Lu2}; in particular, $\breve{\cD}\subseteq \cI$. 
\end{rem}

\begin{exmp} \label{tauE8} Let $W$ be of type $E_8$. The set $\breve{\cD}$ 
is directly available in {\sf PyCox} via the command {\tt distinva}, which 
returns a pair of lists: the first one containing the elements $w\in
\breve{\cD}$, the second one containing the corresponding values $\ba(w)$. 
(This uses pre-stored data from the output of the original computation in 
Example~\ref{twoE8} and, hence, just takes a few seconds.) We now 
decompose $\breve{\cD}$ into $\tau$-cells:
\begin{verbatim}
    >>> W=coxeter("E", 8)
    >>> gt=gentaucells(W, distinva(W)[0]); len(gt)
    81901 
\end{verbatim}
(This computation takes a few days.) Thus, we obtain the partition 
$\breve{\cD}=\breve{\cD}_1 \amalg \ldots \amalg \breve{\cD}_{81901}$ 
where each $\breve{\cD}_i$ is a $\tau$-cell of $\breve{\cD}$. Using the 
knowledge of the values $\ba(w)$ for all $w\in\breve{\cD}$, we can now 
check that the following property holds: 
\begin{equation*}
\mbox{For each $i\in\{1,2,3,\ldots,81901\}$, the function $w\mapsto \ba(w)$ 
is injective on $\breve{\cD}_i$}.\tag{$*$}
\end{equation*}
Since the function $w\mapsto \ba(w)$ is constant on left cells (see
Definition~\ref{aval}) and since elements in the same left cell have the
same generalized $\tau$-invariant (see Corollary~\ref{cortau}), we conclude
that each left cell of $W$ contains a unique element of $\breve{\cD}$;
in particular, there are precisely $101796$ left cells of $W$. This yields
a new proof of the following result:
\end{exmp}

\begin{thm}[Chen \protect{\cite{chene8}}] \label{thm2} Let $W$ be of type 
$E_8$ and $w,w'\in W$. Then $w\sim_L w'$ if and only if $\ba(w)=\ba(w')$ 
(see Definition~\ref{aval}) and $w,w'$ have the same generalized 
$\tau$-invariant (see Definition~\ref{deftau}).  
\end{thm}

\begin{proof} Let $w,w\in W$. If $w\sim_L w'$, then $w,w'$ belong to the 
same two-sided cell and so $\ba(w)=\ba(w')$; furthermore, $w,w'$ have the
same generalized $\tau$-invariant by Corollary~\ref{cortau}. Conversely,
assume that $\ba(w)=\ba(w')$ and that $w,w'$ have the same generalized 
$\tau$-invariant. By Proposition~\ref{pyc56}(b), there exist $d,d\in 
\breve{\cD}$ such that $w\sim_L d$ and $w'\sim_L d'$. Then $\ba(d)=\ba(w)$
and $\ba(w')=\ba(d')$ (by Definition~\ref{aval}) and so $\ba(d)=\ba(d')$.
Furthermore, by Corollary~\ref{cortau}, the elements $d,d'$ have the same 
generalized $\tau$-invariant and so $d,d'\in\breve{\cD}_i$ for some~$i$. 
Hence, Example~\ref{tauE8}($*$) implies that $d=d'$ and so $w\sim_L d
\sim_L w'$. 
\end{proof}

\begin{cor} \label{twoE8a} Let $W$ be of type $E_8$ and $w,w'\in W^*$,
where $W^*$ is defined in Example~\ref{twoE8}. Then $w\sim_{L} w'$ if and
only if $w,w'$ have the same generalized $\tau$-invariant and if there
exists some $E_0\in\cS(W)$ such that $c_{w,E_0}\neq 0$ and $c_{w',E_0}
\neq 0$.
\end{cor}

\begin{proof} Let $w,w'\in W^*$. If $w\sim_L w'$, then $w,w'$ have
the same generalized $\tau$-invariant. Furthermore, by
Example~\ref{twoE8}, there exist $E_0,E_0'\in\cS(W)$ such that
$c_{w,E_0}\neq 0$ and $c_{w',E_0'}\neq 0$. Then $w\in\cF_{E_0}$ and
$w'\in \cF_{E_0'}$. Since $w\sim_L w'$, we conclude that 
$\cF_{E_0}=\cF_{E_0'}$ and so $E_0=E_0'$, using Theorem~\ref{canrep}.
Conversely, assume that $w,w'$ have the same generalized $\tau$-invariant
and that there exists some $E_0\in\cS(W)$ such that $c_{w,E_0}\neq 0$ and 
$c_{w',E_0}\neq 0$. Then $w,w'\in\cF_{E_0}$ and so $\ba(w)=\ba(w')$.
Hence, we have $w\sim_L w'$ by Theorem~\ref{thm2}.
\end{proof}

\begin{exmp} \label{allE8} Let $W$ be of type $E_8$. Because of the sheer 
size of $W$, it would be impossible to determine directly all values
$\ba(w)$ for $w\in W$ and the partition of $W$ into $\tau$-cells. Instead 
we apply the procedure in Remark~\ref{defrel1} in order to obtain a 
set $\Upsilon_I(W)$ and then a set of left cells $\fCla(W) \subseteq
\Upsilon_I(W)$ as in Remark~\ref{klstar1}. For this purpose, let $I$ be 
such that $W_I$ is of type $E_7$. The set $X_I$ of coset representatives 
is obtained by the {\sf PyCox} command {\tt redleftcosetreps}; we have 
$|X_I|=240$. The {\sf PyCox} command 
{\tt klcells} shows that $|\fCl(W_I)|=6364$ and $|\fCla(W_I)|=56$ (see 
\cite[Table~2, p.~246]{pycox}); furthermore, the largest left cell of 
$W_I$ has size $1024$. Hence, we will have to deal with $56$ sets 
$X_I\Gamma_i$ each of which contains at most $240\cdot 1024$ elements. 
For sets of this size, the partition into $\tau$-cells is readily 
determined using the {\tt gentaucells} command of {\sf PyCox}. Explicitly, 
we find that $|\Upsilon_I(W)|=4305120$ and that $\Upsilon_I(W)$ decomposes 
into $614$ $\tau$-cells. Now let $\cC\subseteq \Upsilon_I(W)$ be such a
$\tau$-cell. The decomposition of $\cC$ into left cells is determined as
follows.
\begin{itemize}
\item By Example~\ref{tauE8}, $|\cC\cap \breve{\cD}|$ is the number
of left cells contained in $\cC$. Hence, if $|\cC\cap \breve{\cD}|=1$,
then $\cC$ is a left cell. This applies to $522$ of the $\tau$-cells
in $\Upsilon_I(W)$.
\end{itemize}
Now assume that $|\cC\cap\breve{\cD}|>1$. Let $A:=\{\ba(w) \mid w\in \cC\}$ 
and set $\cC_{(i)}:=\{w\in\cC\mid \ba(w)=i\}$ for $i\in A$. By 
Theorem~\ref{thm2}, each set $\cC_{(i)}$ is a left cell of $W$. Thus, if we 
can determine $\ba(w)$ for all $w\in \cC$, then we can decompose 
$\cC$ into left cells. By Remark~\ref{leftkl}, $\cC$ is a union of left star
orbits. So it suffices to determine $\ba(w)$ for one element $w$ in each 
such orbit. If the orbit contains an involution~$w$, then we obtain 
$\ba(w)$ from the results in Example~\ref{twoE8}. It turns out that, 
inside the $92$ $\tau$-cells $\cC$ such that $|\cC\cap\breve{\cD}|>1$, there 
are $561$ left star orbits which do not contain an involution. For 
representatives of these $561$ orbits, we then have to rely on the techniques 
described in Remark~\ref{remHY} in order to determine $\ba(w)$. (These 
computations take a few days.) Eventually, we find that $\Upsilon_I(W)$ 
decomposes into $746$ left cells. Out of these $746$, we further obtain a 
set $\fCla(W)$ consisting of $106$ left cells. 
\end{exmp}

\begin{rem} \label{sprepelm} 
In {\sf PyCox}, the left cells in $\fCla(W)$ are available via the command 
{\tt klcellreps}, which also returns additional information, e.g., the 
decomposition of $[\Gamma]_1$ into irreducibles and the unique $E_0\in 
\cS(W)$ such that $\Gamma\subseteq\cF_{E_0}$, for any $\Gamma\in\fCla(W)$. 
The star orbit of a left cell in $\fCla(W)$ is computed by the {\sf PyCox}
command {\tt cellrepstarorbit}. In type $E_8$, it takes about $32$ hours 
and a computer with $128$~GB of main memory to produce the complete list 
of all $101796$ left cells of $W$ out of the $106$ left cells in $\fCla(W)$. 
However, just using the set $\fCla(W)$ alone, we now obtain an efficient 
algorithm for determining, for {\em any} $w\in W$, the unique $E_0\in\cS(W)$ 
such that $w\in\cF_{E_0}$. Indeed, given $w$, we compute $R^*(w)$ as defined 
in Remark~\ref{leftkl}. By the definition of $\fCla(W)$, there exists some 
$\Gamma\in\fCla(W)$ such that $\Gamma\cap R^*(w)\neq \varnothing$; let 
$y\in\Gamma\cap R^*(w)$. Since $R^*(w)$ is contained in
a right cell, we have $w\sim_{LR}y$. The unique $E_0\in\cS(W)$ such that
$y\in\cF_{E_0}$ is already known (since $y\in \Gamma$ and $\Gamma\in 
\fCla(W)$); hence, $w\in\cF_{E_0}$ and $\ba(w)=\ba(y)=\ba_{E_0}$. This 
procedure is implemented in the {\sf Python} command {\tt klcellrepelm}. 
(And this function does not require much main memory on a computer.)
By a slight modification of this procedure, one can construct the 
left cell containing any given element $w\in W$; this is implemented
in the {\sf PyCox} command {\tt leftcellelm}. (See the help menues of
these functions for further information.)
\end{rem}

\begin{rem} \label{string} The results of Garfinkle \cite{gar3} and the
fact that an equivalence like that in Theorem~\ref{thm2} even holds in 
type $E_8$ suggest that something similar should be true for the left cells 
in any finite Coxeter group $W$. The idea would be to strengthen the 
definition of the generalized $\tau$-invariant in Definition~\ref{deftau}
using Lusztig's method of ''strings'' \cite[\S 10]{Lu1}. To explain how 
this works, let us consider again two 
generators $s,t\in S$ but drop the assumption that $st$ has order~$3$. 
Assume that $st$ has order $m\geq 3$. Let $W'\subseteq W$ be the parabolic 
subgroup generated by $s,t$. For any $w\in W$, the coset $W'w$ can be 
partitioned into four subsets: one consists of the unique element $x$ 
of minimal length, one consists of the unique element of maximal length, 
one consists of the $(m-1)$ elements $sx,tsx,stsx,\ldots$ and one consists 
of the $(m-1)$ elements $tx,stx, tstx,\ldots$. As in \cite[10.2]{Lu1}, the 
last two subsets (ordered as above) are called {\em strings}. 
By \cite[10.6]{Lu1}, we can now define an involution
\[ \cD_R(s,t)\rightarrow \cD_R(s,t),\qquad w\mapsto \tilde{w},\]
as follows. Let $w\in\cD_R(s,t)$. Then $w^{-1}$ is contained in a unique 
string $\sigma_{w^{-1}}$ (with respect to $s,t$). Let $i\in\{1,\ldots,
m-1\}$ be the index such that $w^{-1}$ is the $i$th element of 
$\sigma_{w^{-1}}$. Then $\tilde{w}$ is defined to be the element such 
that $\tilde{w}^{-1}$ is the $(m-i)$th element of $\sigma_{w^{-1}}$. 

Now let $\Gamma\subseteq \cD_R(s,t)$ be a left cell. 
Then $\tilde{\Gamma}=\{\tilde{w} \mid w\in \Gamma\}$ also is a left cell 
by \cite[Prop.~10.7]{Lu1} (the assumption on $W$ being ''crystallographic''
is now superfluous, thanks to Elias--Williamson \cite{EW}). Hence, as in
Section~\ref{sec2}, we do not only have $\cR(y)=\cR(w)$ but also 
$\cR(\tilde{y})= \cR(\tilde{w})$ for all $y,w\in \Gamma$. Iterating this
process (exactly as in Definition~\ref{deftau} but now allowing any 
generators $s,t$ such that $st$ has order at least $3$), we obtain a stronger 
version of the generalised $\tau$-invariant; the corresponding equivalence 
classes of $W$ are called {\em $\tilde{\tau}$-cells}. We can now state 
the following variation of Vogan's conjecture \cite[3.11]{voga}.
\end{rem}

\begin{conj} \label{sigmatau} For any finite Coxeter group $W$,
two elements $w,w'\in W$ belong to the same left cell if and only if 
$\ba(w)=\ba(w')$ and $w,w'$ belong to the same $\tilde{\tau}$-cell. 
\end{conj}

Using the same argument as in the proof of Theorem~\ref{thm2} (based on a 
computation as in Example~\ref{tauE8}), we have checked that 
Conjecture~\ref{sigmatau} is true for all classical types $B_n,D_n$ where 
$n\leq 9$, and for all exceptional types including the non-crystallographic 
types $H_3$, $H_4$; in type $F_4$, the left cells are precisely the 
$\tilde{\tau}$-cells. It is even possible to formulate of version of 
Conjecture~\ref{sigmatau} for cells with respect to unequal
parameters; see \cite{prep14}.

\section{Applications} \label{secappl}

Using the results in the previous section, it should now be possible
to answer any concrete question concerning the partition into 
Kazhdan--Lusztig cells for $W$ of type $E_8$.

\begin{exmp} \label{conjE8} This arises from the work of Lusztig 
\cite{LuRat} on rationality properties of unipotent representations.
Let $W$ be of type $E_8$ and $C_0$ be the unique conjugacy class of $W$
whose elements have order $6$ and $|C_0|=4480$. Quite remarkably, we have 
$l(w)=40$ for all $w\in W$ in this case; see \cite[Appendix B.6]{gepf}. 
Lusztig \cite[2.17]{LuRat} observed that
\begin{itemize}
\item $|C_0|=4480$ is equal to the number of left cells in the two-sided
cell $\cF$ of $W$ attached to $C_0$ by the method described in 
\cite[2.17]{LuRat}.
\end{itemize}

\begin{landscape}
\begin{table}[ph!]
\caption{Intersections of two-sided cells with $C_{\text{min}}$ for 
cuspidal classes in type $E_8$} \label{tabinter}
{\scriptsize $\renewcommand{\arraystretch}{0.9} \arraycolsep 2pt
\begin{array}{cccc} 
\hline C & o(w) & |C_{\text{min}}| & |C_{\text{min}} \cap \cF_{E_0}| 
\text{ for $E_0\in \cS(W)$} 
\\\hline 8A_1 & 2 & 1 & 1{*}1_x' \\
2D_4(a_1) & 4 & 15120 & 2100{*}2100_y{\cup} 13020{*}4480_y \\
D_4{+}4A_1 & 6 & 56 & 14{*}567_x'{\cup} 2{*}112_z'{\cup} 40{*}1400_z' \\
4A_2 & 3 & 4480 & 4480{*}4480_y \\
E_8(a_8) & 6 & 4480 & 4480{*}4480_y \\
E_7(a_4){+}A_1 & 6 & 11592 & 174{*}2240_x{\cup} 2944{*}4480_y{\cup} 
128{*}2835_x'{\cup} 1760{*}4200_x{\cup} 1408{*}6075_x{\cup} 
80{*}6075_x'{\cup} 780{*}2800_z{\cup} 786{*}4096_z \\&&& {\cup}
480{*}4200_z{\cup} 590{*}4200_z'{\cup} 348{*}4536_z{\cup} 142{*}4536_z'{\cup} 
1098{*}5600_z{\cup} 874{*}5600_z' \\
2D_4 & 6 & 4070 & 42{*}2100_y{\cup} 14{*}2240_x{\cup} 52{*}2240_x'{\cup} 
1814{*}4480_y{\cup} 64{*}2268_x'{\cup} 270{*}2835_x{\cup} 24{*}2835_x'{\cup} 
206{*}4200_x{\cup} 96{*}4200_x'\\&&&{\cup} 418{*}6075_x{\cup} 
38{*}6075_x'{\cup} 82{*}2800_z{\cup} 222{*}4096_z{\cup} 228{*}4200_z'{\cup} 
106{*}4536_z{\cup} 12{*}5600_z{\cup} 382{*}5600_z'\\
2A_3{+}2A_1 & 4 & 1260 & 208{*}2268_x'{\cup} 466{*}4096_x'{\cup} 
258{*}6075_x'{\cup} 108{*}1400_z'{\cup} 114{*}4536_z'{\cup} 
106{*}5600_z'\\
D_8(a_3) & 8 & 7748 & 14{*}210_x{\cup} 168{*}525_x{\cup} 236{*}567_x{\cup} 
366{*}700_x{\cup} 1614{*}1400_x{\cup} 452{*}2240_x{\cup} 404{*}2268_x
\\&&& {\cup}376{*}4200_x{\cup} 44{*}6075_x{\cup} 92{*}560_z{\cup} 
1908{*}1400_z{\cup} 20{*}2800_z{\cup} 1502{*}3240_z{\cup} 552{*}4096_z\\
D_6{+}2A_1 & 10 & 256 & 102{*}4480_y{\cup} 4{*}2268_x'{\cup} 26{*}4200_x{\cup} 
42{*}6075_x{\cup} 26{*}5600_z{\cup} 56{*}5600_z' \\
2A_4 & 5 & 7952 & 38{*}567_x{\cup} 134{*}1400_x{\cup} 1058{*}2240_x{\cup} 
440{*}4480_y{\cup} 272{*}2268_x{\cup} 668{*}2835_x{\cup} 2242{*}4200_x
\\&&& {\cup}200{*}6075_x{\cup} 480{*}1400_z{\cup} 172{*}2800_z{\cup} 
856{*}3240_z{\cup} 1344{*}4096_z{\cup} 48{*}4536_z \\
E_8(a_6) & 10 & 3370 & 12{*}210_x{\cup} 56{*}567_x{\cup} 198{*}700_x{\cup} 
734{*}1400_x{\cup} 348{*}2240_x{\cup} 42{*}2268_x\\&&&{\cup}
146{*}4200_x{\cup} 122{*}560_z{\cup} 966{*}1400_z{\cup} 610{*}3240_z{\cup} 
136{*}4096_z \\
E_6(a_2){+}A_2 & 6 & 16374 & 310{*}1400_x{\cup} 1110{*}2240_x{\cup} 
1774{*}4480_y{\cup} 510{*}2268_x{\cup} 4536{*}4200_x{\cup} 
2116{*}6075_x{\cup} 124{*}1400_z\\&&&{\cup} 1638{*}2800_z{\cup} 
786{*}3240_z{\cup} 1422{*}4096_z{\cup} 382{*}4200_z{\cup} 626{*}4536_z{\cup} 
1040{*}5600_z \\
E_8(a_3) & 12 & 2696 & 34{*}210_x{\cup} 186{*}567_x{\cup} 364{*}700_x{\cup} 
604{*}1400_x{\cup} 78{*}2240_x{\cup} 380{*}560_z{\cup} 870{*}1400_z{\cup} 
180{*}3240_z \\
A_5{+}A_2{+}A_1 & 6 & 3752 & 84{*}2240_x{\cup} 1148{*}4480_y{\cup} 
356{*}2835_x{\cup} 364{*}4200_x{\cup} 492{*}6075_x{\cup} 
166{*}2800_z\\&&&{\cup} 308{*}4096_z{\cup} 294{*}4200_z{\cup} 
136{*}4536_z{\cup} 54{*}5600_z{\cup} 350{*}5600_z' \\
D_8(a_1) & 12 & 2040 & 52{*}210_x{\cup} 40{*}525_x{\cup} 36{*}567_x{\cup} 
82{*}700_x{\cup} 400{*}1400_x{\cup} 52{*}2240_x{\cup} 158{*}560_z{\cup} 
688{*}1400_z{\cup} 532{*}3240_z \\
D_8 & 14 & 852 & 24{*}210_x{\cup} 12{*}525_x{\cup} 16{*}567_x{\cup} 
120{*}700_x{\cup} 200{*}1400_x{\cup} 112{*}560_z{\cup} 272{*}1400_z{\cup} 
12{*}2800_z{\cup} 84{*}3240_z \\
A_7{+}A_1 & 8 & 2080 & 202{*}1400_x{\cup} 58{*}2240_x{\cup} 200{*}4480_y{\cup} 
206{*}2268_x{\cup} 208{*}4200_x{\cup} 180{*}6075_x\\&&&{\cup}
240{*}2800_z{\cup} 332{*}3240_z{\cup} 374{*}4096_z{\cup} 80{*}5600_z \\
E_7{+}A_1 & 18 & 192 & 4{*}525_x{\cup} 8{*}567_x{\cup} 8{*}700_x{\cup} 
100{*}1400_x{\cup} 8{*}560_z{\cup} 52{*}1400_z{\cup} 12{*}2800_z \\
A_8 & 9 & 2816 & 6{*}210_x{\cup} 78{*}567_x{\cup} 172{*}700_x{\cup} 
648{*}1400_x{\cup} 478{*}2240_x{\cup} 62{*}560_z{\cup} 790{*}1400_z{\cup} 
582{*}3240_z \\
E_8(a_4) & 18 & 732 & 104{*}210_x{\cup} 24{*}567_x{\cup} 114{*}700_x{\cup} 
80{*}1400_x{\cup} 12{*}112_z{\cup} 278{*}560_z{\cup} 120{*}1400_z \\
E_8(a_2) & 20 & 624 & 12{*}35_x{\cup} 202{*}210_x{\cup} 16{*}567_x{\cup} 
32{*}700_x{\cup} 16{*}1400_x{\cup} 72{*}112_z{\cup} 242{*}560_z{\cup} 
32{*}1400_z \\
D_5(a_1){+}A_3 & 12 & 15134 & 38{*}525_x{\cup} 42{*}567_x{\cup} 
258{*}1400_x{\cup} 1648{*}2240_x{\cup} 1012{*}4480_y{\cup} 
620{*}2268_x{\cup} 1144{*}2835_x{\cup} 3082{*}4200_x\\&&&{\cup} 
1318{*}6075_x{\cup} 620{*}1400_z{\cup} 1086{*}2800_z{\cup} 
1410{*}3240_z{\cup} 2388{*}4096_z{\cup} 200{*}4536_z{\cup} 
268{*}5600_z \\
E_6{+}A_2 & 12 & 840 & 32{*}567_x{\cup} 28{*}700_x{\cup} 524{*}1400_x{\cup} 
256{*}1400_z \\
E_8(a_7) & 12 & 2360 & 2{*}567_x{\cup} 40{*}700_x{\cup} 1136{*}1400_x{\cup} 
92{*}2240_x{\cup} 188{*}2268_x{\cup} 32{*}4200_x{\cup} 50{*}560_z{\cup} 
536{*}1400_z{\cup} 64{*}3240_z{\cup} 220{*}4096_z \\
E_7(a_2){+}A_1 & 12 & 1758 & 6{*}525_x{\cup} 10{*}567_x{\cup} 4{*}700_x{\cup} 
492{*}1400_x{\cup} 32{*}4480_y{\cup} 84{*}2268_x{\cup} 40{*}4200_x{\cup} 
106{*}6075_x\\&&&{\cup}10{*}560_z{\cup} 302{*}1400_z{\cup} 196{*}2800_z{\cup} 
140{*}3240_z{\cup} 304{*}4096_z{\cup} 32{*}5600_z \\
E_8(a_1) & 24 & 320 & 12{*}35_x{\cup} 114{*}210_x{\cup} 16{*}700_x{\cup} 
90{*}112_z{\cup} 88{*}560_z \\
D_8(a_2) & 30 & 4996 & 24{*}210_x{\cup} 60{*}525_x{\cup} 32{*}567_x{\cup} 
292{*}700_x{\cup} 840{*}1400_x{\cup} 468{*}2240_x{\cup} 180{*}2268_x{\cup} 
38{*}4200_x{\cup} 34{*}6075_x\\&&&{\cup} 104{*}560_z{\cup} 976{*}1400_z{\cup} 
224{*}2800_z{\cup} 1024{*}3240_z{\cup} 528{*}4096_z{\cup} 44{*}4200_z{\cup} 
128{*}4536_z \\
E_8(a_5) & 15 & 1516 & 2{*}35_x{\cup} 174{*}210_x{\cup} 28{*}567_x{\cup} 
368{*}700_x{\cup} 194{*}1400_x{\cup} 20{*}2240_x{\cup} 44{*}112_z{\cup} 
338{*}560_z{\cup} 336{*}1400_z{\cup} 12{*}3240_z \\
E_8 & 30 & 128 & 14{*}35_x{\cup} 48{*}210_x{\cup} 66{*}112_z \\\hline
\end{array}$}
\end{table}
\end{landscape}

He remarks that ''this suggests that $C_0\subseteq \cF$ and that any left 
cell in $\cF$ contains a unique element of $C_0$''. Using {\sf PyCox}, we 
can confirm that this suggestion is true, as follows.
First, we identify our class $C_0$ in the list of the $112$ classes 
returned by the command {\tt conjugacyclasses}.
\begin{verbatim}
    >>> W=coxeter("E", 8)
    >>> c=conjugacyclasses(W)
    >>> [i for i in range(112) if c['classlengths'][i]==4480 and 
                        W.permorder(W.wordtoperm(c['reps'][i]))==6]
    [10]
    >>> cl=conjugacyclass(W, W.wordtoperm(c['reps'][10]))
    # Size of class: 4480
    >>> set([W.permlength(w) for w in cl])
    40           
\end{verbatim}
The last command shows that all elements in $C_0$ indeed have length $40$. 
Next, we check how $C_0$ is partitioned into $\tau$-cells:
\begin{verbatim}
    >>> len(gentaucells(W, cl))  # This will take almost an hour.
    4480
\end{verbatim}
Thus, all the elements lie in pairwise different $\tau$-cells
and, hence, in pairwise different left cells of $W$. Finally, we check that
all elements of $C_0$ lie in the same two-sided cell.
\begin{verbatim}
    >>> klcellrepelm(W,cl[0])['special']      # see Remark 6.7
    '4480_y'                  
    >>> set([klcellrepelm(W,w)['special'] for w in cl])
    set(['4480_y'])
\end{verbatim}
(This takes about a quarter of an hour.) Thus, we have $C_0\subseteq 
\cF_{4480_y}$. By Table~\ref{fame8}, we also obtain $\ba(w)=16$ for all
$w\in C_0$.

More generally, let $C$ be any conjugacy class of $W$. Let $d_C=
\min\{l(w)\mid w\in C\}$ and $C_{\text{min}}=\{w\in C\mid l(w)=d_C\}$ be
the set of elements of minimal length in $C$; see \cite[\S 3.1]{gepf}. 
Furthermore, we say that $C$ is cuspidal if $C\cap W_I=\varnothing$ for
any proper $I\subsetneqq S$. Thus, the class $C_0$ considered above is
a cuspidal class such that $C_0=C_{0,\text{min}}$. Table~\ref{tabinter} 
shows the cardinalities of the intersections $C_{\text{min}}\cap \cF_{E_0}$
as $E_0$ runs over the set $\cS(W)$ of special representations.
(The notation $n_1*E_1\cup n_2*E_2\cup\ldots$ means $|C_{\text{min}}\cap
\cF_{E_i}|=n_i$ for $i=1,2,\ldots$. The above example $C_0$ corresponds 
to the $5$th row of the table.)
\end{exmp}

Next, we discuss the following conjecture.

\begin{conj}[Lusztig, cf.\ \protect{\cite{shi1}, \cite{xi1}}] \label{conn0} 
Every left cell $\Gamma$ of $W$ is left-connected, that is, for any two 
elements $x,y \in \Gamma$, there is a chain of generators $s_1,s_2,\ldots,
s_n$ in $S$ such that $y=s_n \cdots s_2s_1x$ and all intermediate elements 
$s_1x$, $s_2s_1x$, $\ldots$, $s_{n-1}\cdots s_2s_1x$ lie in $\Gamma$.
\end{conj}

\begin{exmp} \label{conn1} Using the {\sf PyCox} commands {\tt klcellreps} 
and {\tt cellrepstarorbit} (see Remark~\ref{sprepelm}), we have a way of 
running through all the left cells of $W$. Furthermore, it is 
straightforward to write a function which verifies if a given left cell
is left-connected or not. In this way, we have verified that 
Conjecture~\ref{conn0} holds for all $W$ of exceptional type $H_3$, $H_4$, 
$F_4$, $E_6$, $E_7$, $E_8$. (For type $E_8$, this takes about $2$ or $3$ 
days; note that it is not necessary to keep all the left cells at once in 
the main memory of the computer.) For type $A_n$, the conjecture holds by 
\cite[\S 5]{KL}; for $I_2(m)$, it follows easily from the description of 
the left cells in \cite[7.15]{Lusztig03}. For type $B_n$, it follows
from Garfinkle \cite[Theorem~3.5.9]{gar3} (via the known dictionary 
between left cells and the corresponding notions in the theory of 
primitive ideals in enveloping algebras; see \cite[5.25]{LuBook} and the 
references there). The question seems to be open for $D_n$. 
\end{exmp}

Finally, we come to Kottwitz' conjecture \cite{kottwitz}. Let
$W$ be a finite Coxeter group and $C$ be a conjugacy class of 
involutions in $W$. Following \cite[\S 1]{kottwitz}, \cite[6.3]{luvo}, 
let $V_{C}$ be an $\R$-vector space with a basis $\{a_w \mid w \in C\}$. 
Then there is a linear action of $W$ on $V_{C}$ such that, for any 
$s \in S$ and $w \in C$, we have
\[ s.a_w=\left\{\begin{array}{cl} -a_w & \qquad \mbox{if $sw=ws$ and
$\ell(sw)<\ell(w)$},\\ a_{sws} & \qquad \mbox{otherwise}.\end{array}\right.\]

\begin{conj}[Kottwitz \protect{\cite[\S 1]{kottwitz}}] \label{coko}
Let $C$ be a conjugacy class of involutions and $\Gamma$ be a left 
cell of $W$. Then $\dim \Hom_W(V_C,[\Gamma]_1)=|C \cap \Gamma|$.
\end{conj}

By work of Kottwitz himself, Casselman \cite{cass}, Bonnaf\'e and the 
first-named author \cite{boge}, \cite{pycox}, \cite{tkott}, this conjecture 
is already known to hold except possibly for $W$ of type $E_8$. The 
verification for type $E_8$ is now a matter of combining various pieces of 
known information. The decompositions of the representations $V_C$ into 
irreducibles can be computed using the known character table of $W$; see 
\cite{cass}. In {\sf PyCox}, this is done using the command 
{\tt involutionmodel}. Now let $\Gamma$ be a left cell. By 
Example~\ref{tauE8}, there is a unique element $d\in\Gamma\cap \breve{\cD}$;
we then write $\Gamma=\Gamma_d$. By Proposition~\ref{pyc56} and 
Example~\ref{twoE8}(a), we have $m(\Gamma_d,E) =c_{d,E}$ for all 
$E\in\Irr(W)$. Hence, we have 
\[\dim\Hom_W(V_C,[\Gamma_d]_1)=\sum_{E\in\Irr(W)} c_{d,E}
\dim\Hom_W(V_C,E),\]
and these dimensions can be explicitly determined using the results
in Example~\ref{twoE8} and the output of {\tt involutionmodel} (or the 
tables in \cite{cass}). Finally, by Example~\ref{twoE8}(c), there is a 
unique $E_0\in\cS(W)$ such that $c_{d,E_0}\neq 0$. Then, by 
Corollary~\ref{twoE8a}, we have
\begin{align*}
C\cap \Gamma_d&=\{w\in C\mid \ba(w)=\ba(d) \mbox{ and } d,w 
\mbox{ belong to the same $\tau$-cell}\}\\ &=\{w\in C\mid c_{w,E_0}\neq 0 
\mbox{ and } d,w \mbox{ belong to the same $\tau$-cell}\}.
\end{align*}
These intersections can be determined explicitly using the
{\sf PyCox} command {\tt gentaucells} (applied to $C$) and the results 
in Example~\ref{twoE8} (or the command {\tt klcellrepelm} in 
Remark~\ref{sprepelm}). In this way, we have verified that 
Conjecture~\ref{coko} holds for $W$ of type $E_8$.

\medskip
\noindent {\bf Acknowledgements.} We thank G. Lusztig for pointing out 
the suggestion in \cite[2.17]{LuRat}. We are also indebted to T. Pietraho 
who helped to clarify some points about Garfinkle's work \cite{gar3}. 
 

\end{document}